\documentclass[10pt]{amsart}
\usepackage[english]{babel}
\usepackage{amssymb}
\usepackage{amsmath}

\newcommand{\can}{\overline{\phantom{x}}}

\def\Cay{{\rm{Cay}}}

\newtheorem{dummy}{Dummy}

\newtheorem{lemma}[dummy]{Lemma}
\newtheorem{theorem}[dummy]{Theorem}
\newtheorem{proposition}[dummy]{Proposition}
\newtheorem{corollary}[dummy]{Corollary}

\theoremstyle{definition}

\newtheorem{example}[dummy]{Example}

\newtheorem{remark}[dummy]{Remark}

\subjclass[2000]{Primary: 17A35. }

\date{8.1.2010}
\author{S. Pumpl\"un}
\email{susanne.pumpluen@nottingham.ac.uk}
\address{School of Mathematical Sciences\\
University of Nottingham\\
University Park\\
Nottingham NG7 2RD\\
United Kingdom
}

\begin{document}

\title[division algebras from generalized Cayley-Dickson doublings]
{How to obtain division algebras from a generalized Cayley-Dickson doubling process}
\maketitle

\begin{abstract}
New families of eight-dimensional real division algebras with large derivation algebra
 are presented: We generalize the classical Cayley-Dickson doubling process
starting with a unital algebra  with involution over a field $F$ by allowing the scalar in the doubling to be an invertible
 element in the algebra. The resulting unital algebras are neither power-associative nor quadratic.
 Starting with a quaternion division algebra $D$, we obtain
 division algebras $A$ for all invertible scalars chosen in $D$ outside of $F$. This is independent on where the
 scalar is placed inside the product and three pairwise non-isomorphic families of eight-dimensional division
 algebras are obtained.
 Over $\mathbb{R}$, the derivation algebra of each such algebra $A$ is isomorphic to $ su(2)\oplus F$
and the decomposition of  $A$ into irreducible $su(2)$-modules has  the form
$1+1+3+3$ (denoting an irreducible $su(2)$-module by its dimension).
 Their opposite algebras yield more classes of pairwise non-isomorphic families
 of division algebras of the same type. We thus give an affirmative answer to a question posed by Benkart and Osborn
 in 1981.
 \end{abstract}

\section*{Introduction}

 It is well-known that every real division algebra must have dimension 1, 2, 4 or 8 [M-B, K], and, indeed, the same result applies to
 division algebras over a real closed field [D-D-H].

Real division algebras were roughly classified by Benkart and Osborn [B-O1, 2]
according to the isomorphism type of their derivation algebra.
In the introduction of [B-O2], it is noted that ``as one might expect, most of the classes of division algebras are
 natural generalizations
of the quaternions and octonions,'' with the only exception of one family of flexible algebras which includes Okubo's
 pseudo-octonion algebras. The authors  constructed division algebras over fields of characteristic 0 by slightly manipulating
the usual octonion multiplication. A list of 5 types of possible derivation algebras
${\rm Der}(A)$ was given. Using the representation theory of Lie algebras,
 for each type of derivation algebra families of real division algebras of dimension 8
were investigated which display this non-zero Lie  algebra as their derivation algebra in [B-O1, 2].
This work was continued in [R] and [Do-Z1, 2].

Flexible real division algebras were classified in subsequent papers by Benkart, Britten, Osborn and Darp\"{o}
 [B-B-O, D1, D2], see also Cuenca Mira, De Los Santos Villodres, Kaidi and Rochdi [C-V-K-R].
 Ternary derivations were used in [J-P] to describe large classes of division algebras. Power-associative
 real division algebras of dimension less or equal to 4 were classified in [Die], to list
 just of a few of the other known results.

 However, the problem to find all real division algebras which are neither alternative nor flexible
remains open.

 Let $F$ be a field. A {\it nonassociative
quaternion algebra} over $F$ is a four-dimensional unital $F$-algebra $A$ whose
nucleus is a quadratic \'etale algebra over $F$.  Nonassociative quaternion division algebras
canonically appeared as the most interesting case in the classification of the
algebras of dimension 4 over $F$ which contain
a separable field extension $K$ of $F$ in their nucleus (Waterhouse [W], see also
Althoen-Hansen-Kugler [Al-H-K] for $F=\mathbb{R}$). They
were first discovered by Dickson [Di] in 1935
  and Albert [A] in 1942 as early examples of real division algebras.
Waterhouse's classification shows that every nonassociative quaternion
algebra is a Cayley-Dickson doubling of its nucleus $S$,
where the scalar chosen for this doubling process is an invertible element in $S$ not contained
in the base field, see also [As-Pu].

In this paper, we extend this idea and construct algebras  over $F$ using the Cayley-Dickson doubling of a
unital algebra with involution, where the scalar chosen for the doubling process is an invertible element of the algebra which
is not contained in the base field.
 The description of these algebras is straightforward and base free.
  They are unital, not power-associative and not quadratic.
They do not even satisfy the weaker property that $u^2u=uu^2$ holds for all elements $u$, i.e. they are not
third-power associative.
 We call them {\it Dickson algebras}.
Most of the time, we look at the Cayley-Dickson doubling of a quaternion algebra
$D$ over  $F$.
There are three different possibilities  how to place the scalar $c\in D^\times\setminus F$
used in the doubling process inside the product. If $c$ lies in a separable quadratic subfield of a quaternion
division algebra $D$, each possibility yields a different family of algebras.
 Over $\mathbb{R}$,  for instance, there are the pairwise non-isomorphic division
 algebras ${\rm Cay}(\mathbb{H},i)$, ${\rm Cay}_m(\mathbb{H},i)$ and ${\rm Cay}_r(\mathbb{H},i)$.

The new eight-dimensional algebras ${\rm Cay}(D,c)$,
${\rm Cay}_m(D,c)$ and ${\rm Cay}_r(D,c)$ contain both quaternion and nonassociative quaternion algebras
 as subalgebras (unless $c\in D^\times$ is contained in a purely inseparable quadratic field extension $K$
of $F$) and are division algebras if and only if $D$ is a division algebra.
If $D$ is a quaternion division algebra then
$D$ is the only quaternion subalgebra of  ${\rm Cay}(D,c)$, ${\rm Cay}_m(D,c)$ and ${\rm Cay}_r(D,c)$.
Eight-dimensional Dickson  algebras are central.

 Properties of these families of eight-dimensional algebras are studied in Section 2, with the scalar alternatively placed in the middle or the left or
  right-hand-side of the relevant product.
In Section 3 we investigate when two eight-dimensional Dickson algebras are isomorphic and in Section 4 we look at their derivations.
 If  $F=\mathbb{R}$, the derivation algebra of $A={\rm Cay}(D,c)$, $A={\rm Cay}_m(D,c)$ and $A={\rm Cay}_r(D,c)$
 is isomorphic to $ su(2)\oplus F$ and the decomposition of  $A$ into irreducible $su(2)$-modules has the form
$1+1+3+3$. Thus these algebras yield new examples of real division algebras with large
derivation algebra and so do their opposite algebras, which are no longer Dickson algebras.
 To our knowledge, these are the first examples of algebras with derivation algebra isomorphic to
 $ su(2)\oplus F$ where the decomposition of $A$ into irreducible
$ su(2)$-modules has the form $1+1+3+3$. The known examples where the decomposition of $A$ into irreducible
$ su(2)$-modules has the form $1+1+3+3$ given by Rochdi [R] are real noncommutative Jordan algebras
 with derivation algebra isomorphic to $ su(2)$.
 We thus answer one of the questions posed by Benkart and Osborn [B-O1] in the affirmative.

In Section 5 we briefly investigate their opposite algebras which constitute three other new
classes of division algebras with derivation algebra isomorphic to
 $ su(2)\oplus F$ where the decomposition of $A$ into irreducible
$ su(2)$-modules has the form $1+1+3+3$,
 and in Section 6 we briefly investigate classes of 16-dimensional (division) algebras over suitable base fields.

\section{Preliminaries}

\subsection{Nonassociative algebras}

Let $F$ be a field. By ``$F$-algebra'' we mean a finite dimensional unital nonassociative algebra over $F$.

 A nonassociative algebra $A$ is called a {\it division algebra} if for any $a\in A$, $a\not=0$,
the left multiplication  with $a$, $L_a(x)=ax$,  and the right multiplication with $a$, $R_a(x)=xa$, are bijective.
$A$ is a division algebra if and only if $A$ has no zero divisors [Sch, pp. 15, 16].

For an $F$-algebra $A$, associativity in $A$ is measured by the {\it associator} $[x, y, z] = (xy) z - x (yz)$.
The {\it left nucleus} of $A$ is defined as $N_l(A) = \{ x \in A \, \vert \, [x, A, A]  = 0 \}$, the
{\it middle nucleus} of $A$ is
defined as $N_m(A) = \{ x \in A \, \vert \, [A, x, A]  = 0 \}$ and  the {\it right nucleus} of $A$ is
defined as $N_r(A) = \{ x \in A \, \vert \, [A,A, x]  = 0 \}$.
Their intersection, the {\it nucleus} of $A$ is then
$N(A) = \{ x \in A \, \vert \, [x, A, A] = [A, x, A] = [A,A, x] = 0 \}$. The nucleus is an associative
subalgebra of $A$ (it may be zero) and $x(yz) = (xy) z$ whenever one of the elements $x, y, z$ is in
$N(A)$.

The {\it commuter} of $A$ is defined as ${\rm Comm}(A)=\{x\in A\,|\,xy=yx \text{ for all }y\in A\}$
and the {\it center} of $A$ is ${\rm C}(A)=\{x\in A\,|\, x\in \text{Nuc}(A) \text{ and }xy=yx \text{ for all }y\in A\}$.

Let $S$ be a quadratic \'etale algebra over $F$ (i.e., a separable quadratic $F$-algebra
in the sense of [Knu, p.~4])
with canonical involution $\sigma \colon S \to S$, also written as
$\sigma =^-$, and with nondegenerate norm
$N_{S/F} \colon S \to S$, $N_{S/F} (s) =s \overline{s} = \overline{s} s$.
$S$ is a two-dimensional unital commutative associative algebra over $F$. With the diagonal action of $F$, $F \times F$ is a quadratic \'etale algebra with
canonical involution $(x, y) \mapsto (y, x)$.

\subsection{Nonassociative quaternion division algebras}

Let $F$ be a field. A {\it nonassociative quaternion algebra} over $F$ is a four-dimensional unital $F$-algebra $A$ whose
nucleus is a quadratic \'etale algebra over $F$. Let $S$ be a quadratic \'etale algebra over $F$ with
canonical involution $\sigma =^-$. For every $b \in S \setminus F$, the
vector space
$$\Cay (S, b) = S \oplus S$$
becomes a nonassociative quaternion algebra over $F$ with unit element $(1,0)$ and nucleus $S$ under the multiplication
$$(u, v) (u', v') = (u u' + b \overline{v}' v, v' u + v \overline{u}')$$
for $u, u', v, v' \in S$. Given any nonassociative quaternion algebra
$A$ over $F$ with nucleus $S$, there exists an element $b\in S^{\times}\setminus F$ such that $A\cong\Cay (S,b)$
[As-Pu, Lemma 1].

The multiplication is thus defined analogously as for a quaternion algebra, only that we require the scalar $b$ to lie outside of $F$.
Nonassociative quaternion algebras are neither power-associative nor quadratic and
 $\Cay (S, b)$ is a division algebra if and only if
$S$ is a separable quadratic field extension of $F$ [W, p.~369].

\begin{remark}
 Let $F$ have characteristic not 2 and $K=F(\sqrt{a})=F(i)$ for some $a\in F$ be a quadratic field extension with norm $N_{K/F}(x)=x\sigma(x)$,
 where $\sigma$ is the non-trivial automorphism of $K$ that fixes $F$. Let $b\in K\setminus F$, so that
 $A={\rm Cay}(K,b)$ is a nonassociative quaternion division algebra. Put
 $j=(0,1)\in {\rm Cay}(K,b).$
Then $A$ has $F$-basis $\{1,i,j,ij\}$ such that $i^2=a$, $j^2=b$ with $b\in K\setminus F$ and
 $jx=\sigma(x)j$
 (so in particular $ij=-ji$) for all $x\in K$.
\end{remark}

\begin{remark} Let $F=\mathbb{R}$ and $b,b'\in \mathbb{C}\setminus \mathbb{R}$. Then ${\rm Cay} (\mathbb{C}, b) \cong {\rm Cay} (\mathbb{C}, b')$
if and only if $b'=tb$ or $b'=t \bar b$ for some positive $t\in \mathbb{R}$  [Al-H-K, Thm. 14].
Two nonassociative quaternion  algebras ${\rm Cay}(K,b)$ and
${\rm Cay}(L,c)$ can only be isomorphic if $L=K$. Moreover,
$${\rm Cay}(K,b)\cong {\rm Cay}(K,c) \text{ iff } g(b)=N_{K/F}(d)c$$
for some automorphism $g\in {\rm Aut}(K)$ and some non-zero $d\in K$ [W].
\end{remark}

If $A$ is a real division algebra of dimension 4, then its derivation algebra is isomorphic to $su(2)$ or
${\rm dim}\, {\rm Der}A=0$ or 1 by the classification theorem in [B-O1].

\begin{theorem} ([W, Thm. 3]) Let $K$ be a separable field extension of $F$.
The derivations of the nonassociative quaternion division algebra $A={\rm Cay}(K,b)$ have the form
$D((u,v))=(0,sv)$ with $s\in K$ such that $s+\sigma(s)=0$. In particular, ${\rm Der}A\cong F$.
\end{theorem}

\section{The generalized Cayley-Dickson doubling process with different positions for the scalar $c$}

Let $D$ be a unital algebra over $F$ with an involution $\sigma:D\to D$. Let  $c\in D$ be an
invertible element not contained in $F$ such that $\sigma(c)\not=c$. Then the $F$-vector space
 $A=D\oplus D$ can be made into a unital algebra over $F$ via the multiplications
\begin{enumerate}
\item $\quad\quad\quad (u,v)(u',v')=(uu'+c (\sigma(v')v),v'u+v\sigma(u'))$
\item $\quad\quad\quad (u,v)(u',v')=(uu'+ \bar v'(cv),v'u+v\bar u')$\\
or
\item
$\quad\quad\quad (u,v)(u',v')=(uu'+ (\bar v'v)c,v'u+v\bar u')$
\end{enumerate}
for $u,u',v,v'\in D$.
If the algebra $D$ is not associative, we can  also study the multiplications
\begin{enumerate}
\item[(1*)] $\quad\quad\quad (u,v)(u',v')=(uu'+(c \sigma(v'))v,v'u+v\sigma(u'))$
\item[(2*)] $\quad\quad\quad (u,v)(u',v')=(uu'+ \bar (v'c)v,v'u+v\bar u')$\\
or
\item[(3*)]
$\quad\quad\quad (u,v)(u',v')=(uu'+ \bar v'(vc),v'u+v\bar u')$
\end{enumerate}
for $u,u',v,v'\in D$.
 The unit element of the new algebra $A$ is given by $1=(1,0)$ in each case.

$A$ is called the {\it Cayley-Dickson doubling of $D$ (with scalar $c$ on the left hand side,
 in the middle, or on the right hand side)} and denoted by ${\rm Cay}(D,c)$ for multiplication
(1), by ${\rm Cay}_m(D,c)$ for multiplication (2) and  by ${\rm Cay}_r(D,c)$ for multiplication (3).
If $D$ is not associative, we denote the algebra $A$ by  ${\rm Cay}^l(D,c)$ for multiplication
(1*), by ${\rm Cay}^l_m(D,c)$ for multiplication (2*) and  by ${\rm Cay}^r_r(D,c)$ for multiplication (3*).
We call every such algebra obtained from a Cayley-Dickson doubling of $D$, with the scalar $c$
 in the middle, resp. on the left or right hand side,  a {\it Dickson algebra} over $F$.

We can immediately say the following:

\begin{lemma}
Let $A={\rm Cay}(D,c)$, resp. $A={\rm Cay}_m(D,c)$, $A={\rm Cay}_r(D,c)$,
$A={\rm Cay}^l(D,c)$, $A={\rm Cay}^l_m(D,c)$ or $A={\rm Cay}^r_r(D,c)$.\\
(i) $D$ is a subalgebra of $A$.\\
(ii) If there is a unital subalgebra $(B,\sigma)$ of $ ( D,\sigma)$ with involution  such that $c\in B^\times$ then
 the Dickson algebra ${\rm Cay}(B,c)$, resp. ${\rm Cay}_m(B,c)$, ${\rm Cay}_r(B,c)$,
${\rm Cay}^l(B,c)$, ${\rm Cay}^l_m(B,c)$ or ${\rm Cay}^r_r(B,c)$, is a subalgebra of $A$.
\\ (iii) $A$ is not third power-associative and not quadratic.
\end{lemma}

\begin{proof} (i) and (ii) are trivial.
\\ (iii) For $l=(0,1)$ we have $l^2=(c,0)$ and $ll^2=(0,\sigma(c))$ while $l^2l=(0,c)$, hence $A$ is not third
power-associative. Every quadratic unital algebra is clearly power-associative, so
$A$ is not quadratic.
\end{proof}

From now on, unless otherwise specified,
$$\text{ let } D \text{ be a quaternion algebra over } F.$$
 Let $\sigma=\can:D\to D$ be the
canonical involution of $D$ and $c\in D^\times\setminus F$.
Let $A={\rm Cay}(D,c)$, $A={\rm Cay}_m(D,c)$ or $A={\rm Cay}_r(D,c)$.

\begin{remark} (i) Let $D$ be a division algebra over $F$. Then every element $c\in D$ not contained in $F$
lies in some quadratic  field extension $K\subset D$ of $F$ which is either purely inseparable or cyclic [Dr, p.~160].
\\ (ii)  Let $c\in D^\times\setminus F$.
 If $F$ has characteristic not 2, there is a separable quadratic  field extension $K\subset D$ such that $c\in K$
and the nonassociative quaternion algebra ${\rm Cay}(K,c)$ is a subalgebra of $A$.\\
Suppose $F$ has characteristic 2.
If $c$ lies in a purely inseparable extension $K'\subset D$ then ${\rm Cay}(K',c)$ is a subalgebra of $A$.
\end{remark}

\begin{remark} Let $F$ have characteristic not 2 and let $D=(a,b)_F$.
 $(a,b)_F$ has the standard $F$-basis $\{1,i,j,k\}$ such that $i^2=a$, $j^2=b$, $ij=k$ and
 $ij=-ji$. Let $l=(0,1)\in D\oplus D={\rm Cay}(D,c)$.
 Let $c\in D^\times\setminus F$. The Dickson algebras ${\rm Cay}(D,c)$, ${\rm Cay}_m(D,c)$ and ${\rm Cay}_r(D,c)$ have
  the standard $F$-basis
 $$\{1,i,j,k,l,il,jl,kl\}$$
 with  $l^2=c$. Note that $ul=l\sigma(u)$ for all $u\in D$.
\end{remark}

\begin{lemma}
(a) Let $A={\rm Cay}(D,c)$. Then
 ${\rm Nuc}_l(A)={\rm Nuc}_m(A)={\rm Nuc}_r(A)=F$, hence  ${\rm Nuc}(A)=F$.
Furthermore, ${\rm Comm}(A)=F$ and ${\rm C}(A)=F.$
\\ (b) Let $A={\rm Cay}_m(D,c)$. Then ${\rm Nuc}_l(A)={\rm Nuc}_m(A)={\rm Nuc}_r(A)=F$, thus ${\rm Nuc}(A)=F$.
  If $D$ is a division algebra then ${\rm Comm}(A)=F$, else
$${\rm Comm}(A)=\{(u,v)\,|\, u\in F \text{ and } v(\overline{u'}-u')=0, \overline{v'} cv=\bar vcv'
\text{ for all } u',v'\in D \}$$
$$\subset \{(u,v)\,|\, u\in F,  cv=\bar v c ,  N_{D/F}(v)=0 \}.$$
Thus ${\rm C}(A)=F$ in any case.
\\ (c)  Let $A={\rm Cay}_r(D,c)$. Then
 ${\rm Nuc}_l(A)={\rm Nuc}_m(A)={\rm Nuc}_r(A)=F$, hence ${\rm Nuc}(A)=F$,
and ${\rm Comm}(A)=F$ and ${\rm C}(A)=F.$
\end{lemma}

The proofs  are straightforward but tedious computations.

Over a field of characteristic 2, we define the quaternion algebra $[a,b)$ over $F$ via
$[a,b)=\langle i,j\,|\, i^2+i=a,j^2=b,ij=ji+j \rangle$
with $a\in F$ and $b\in F^\times$. Obviously,  $[a,b)={\rm Cay}(L,b)$ with $L=F(i)$ where
$i^2+i=a$ is a separable quadratic field extension [S, p.~314].

\begin{theorem} Let  $D$ be a division algebra and $c\in D^\times\setminus F$. Then
$D$ is the only quaternion subalgebra of the Dickson algebra $A={\rm Cay}(D,c)$.
\end{theorem}

\begin{proof} We distinguish two cases.\\
 (i)  Let $F$ have characteristic not 2. Suppose there is a quaternion subalgebra $B=(e,f)_F$ in $A$. Then there is an element $X\in A$, $X=(u,v)$
with $u,v\in D$ such that $X^2=e\in F^\times$ and an element $Y\in A$, $Y=(w,z)$
with $w,z\in D$ such that $Y^2=f\in F^\times$ and $XY+YX=0$. The first equation implies
$$u^2+cN_{D/F}(v)=e \text{ and } v(u+\sigma(u))=0.$$
If $v=0$, then $u^2=e$ and $X=(u,0)\in D$.\\
If $v\not=0$ then $v$ is invertible and $\sigma(u)=-u$. This implies $u^2=-N_{D/F}(u)$, thus $e+N_{D/F}(u)=cN_{D/F}(v)$ i.e.
$N_{D/F}(v^{-1})(e+N_{D/F}(u))=c$ which is a contradiction since
the right hand side lies in $D$ and not in $F$, while the left hand side lies in $F$.\\
Analogously, the second equation implies $w^2=f$ and $Y=(w,0)$, $w\in D$. Hence $(0,0)=XY+YX=(uw+wu,0)$ means
$uw+wu=0$ and so the standard basis $1,X=u,Y=w,XY=uw$ for the quaternion algebra $(e,f)_F$ lies in $D$ and we obtain
$D=(e,f)_F$.\\
(ii) Let $F$ have characteristic 2.
 Suppose there is a quaternion subalgebra $B=[e,f)$ in $A$.
Then there is an element $X\in A$, $X=(u,v)$
with $u,v\in D$ such that $X^2+X=e\in F$ and an element $Y\in A$, $Y=(w,z)$
with $w,z\in D$ such that $Y^2=f\in F^\times$ and $XY=YX+Y$.
Analogously as in the above proof, the equation $Y^2=f\in F^\times$ implies $w^2=f$ and $Y=(w,0),$ $ w\in D$.

The first equation implies
$$u^2+cN_{D/F}(v)+u=e \text{ and } v(u+\sigma(u)+1)=0.$$
If $v\not=0$ then $v$ is invertible and $\sigma(u)+u+1=0$.
 This implies $u=-(\sigma(u)+1)$, thus $\sigma(u^2)+\sigma(u)+cN_{D/F}(v)=e$ i.e.
we get $\sigma(u^2+u)+cN_{D/F}(v)=u^2+u+cN_{D/F}(v)$ which yields $\sigma(u^2+u)=u^2+u$. Hence
$u^2+u\in F.$ This implies that $cN_{D/F}(v)=e-(u^2+u)\in F $, a contradiction.

Hence  $v=0$ which implies $u^2+u=e$ and $X=(u,0)\in D$. Now $XY=YX+Y$ means
$uw=wu+w$ and the standard basis $1,X=u,Y=w$ for the quaternion algebra $[e,f)$ lies in $D$. We obtain
$D=[e,f)$.
\end{proof}

If $D$ is not a division algebra we only get $e+N_{D/F}(u)=cN_{D/F}(v)$ in the proof of (i).
Here it could happen that the right hand side is 0, the left hand side still lies in $F$, so we get no contradiction
unless $v$ is not a zero divisor. A similar argument applies in (ii).

\begin{theorem} Let  $D$ be a division algebra and $c\in D^\times\setminus F$. Then
$D$ is the only quaternion subalgebra of the algebras ${\rm Cay}_m(D,c)$ and ${\rm Cay}_r(D,c)$.
\end{theorem}

\begin{proof}

The proof for $A={\rm Cay}_r(D,c)$ is analogous to the one of Theorem 8, we include the one for
$A={\rm Cay}_m(D,c)$ for the sake of the reader:\\
(i)  Let $F$ have characteristic not 2 and let $D=(a,b)_F$ be a division algebra. Suppose there is a quaternion subalgebra $B=(e,f)_F$ in $A$. Then there is an element $X\in A$, $X=(u,v)$
with $u,v\in D$ such that $X^2=e\in F^\times$ and an element $Y\in A$, $Y=(w,z)$
with $w,z\in D$ such that $Y^2=f\in F^\times$ and $XY+YX=0$. The first equation implies
$$u^2+\bar vcv=e \text{ and } v(u+\bar u)=0.$$
If $v\not=0$ then $v$ is invertible and $\bar u =-u$. This implies $u^2=-N_{D/F}(u)$, thus
$e+N_{D/F}(u)=\bar vcv$, so $\bar vcv\in F$. Thus $N_{D/F}(v)cv \in F v $ (multiply with $v$ from the right) which implies $c\in F$,
a contradiction. Thus $v=0$.
\\
Analogously, the second equation implies $w^2=f$ and $Y=(w,0)$, $w\in D$. Hence $(0,0)=XY+YX=(uw+wu,0)$ means
$uw+wu=0$ and so the standard basis $1,X=u,Y=w,XY=uw$ for the quaternion algebra $(e,f)_F$ lies in $D$ and we obtain
$D\cong(e,f)_F$.\\
(ii)
 Let $F$ have characteristic 2 and let $D=[a,b)$ be a division algebra.
 Suppose there is a quaternion subalgebra $B=[e,f)$ in $A$.
Then there is an element $X\in A$, $X=(u,v)$
with $u,v\in D$ such that $X^2+X=e\in F$ and an element $Y\in A$, $Y=(w,z)$
with $w,z\in D$ such that $Y^2=f\in F^\times$ and $XY=YX+Y$.
Analogously as in (i), the  equation $Y^2=f\in F^\times$ implies $w^2=f$ and $Y=(w,0),$ $ w\in D$.

The first equation implies
$$u^2+\bar v cv+u=e \text{ and } v(u+\sigma(u)+1)=0.$$
If $v\not=0$ then $v$ is invertible and $\sigma(u)+u+1=0$.
 This implies $u=-(\sigma(u)+1)$, thus  thus $\sigma(u^2)+\sigma(u)+\bar v cv=e$ i.e.
we get $\sigma(u^2+u)+\bar v cv=u^2+u+\bar v cv$ which yields $\sigma(u^2+u)=u^2+u$. Hence
$u^2+u\in F.$ This implies that $\bar v cv=e-(u^2+u)\in F $, a contradiction: multiply this equation with $ v$ from the
left to obtain $cv\in Fv$ and thus $c\in F$.

Hence  $v=0$ which implies $u^2+u=e$ and $X=(u,0)\in D$. Now $XY=YX+Y$ means
$uw=wu+w$ and the standard basis $1,X=u,Y=w$ for the quaternion algebra $[e,f)$ lies in $D$. We obtain
$D\cong[e,f)$.
\end{proof}

\begin{theorem}  Let  $D$ be a division algebra and $c\in D^\times$ contained in a separable quadratic field extension
$L$ of $F$. Then ${\rm Cay}(L,c)$ is the only nonassociative quaternion subalgebra of the Dickson algebra
$A={\rm Cay}(D,c)$, $A={\rm Cay}_m(D,c)$ and $A={\rm Cay}_r(D,c)$.
\\ If $c\in D^\times$ is instead contained in a purely inseparable quadratic field extension $K$
of $F$, then $A$ has no nonassociative quaternion subalgebras. If
${\rm Cay}(K',e)$ is a subalgebra of $A$, where $K'$ is a purely inseparable quadratic field extension of
$F$, then ${\rm Cay}(K',e)={\rm Cay}(K,c)$.
\end{theorem}

\begin{proof} Suppose that
${\rm Cay}(K,e)$ is a nonassociative quaternion subalgebra of $A$.
\\ (i) Let $F$ have characteristic not 2 and let $K=F(\sqrt{f})$. Then there is an
element $Y=(w,z)\in A$, $w,z\in D$, such that $Y^2=f\in F^\times$.
Analogously as in the  proof of Theorem 8, resp. 9, this implies $Y=(w,0)$ and $w^2=f$, $w\in D$, so $K\subset D$
and $K\oplus K\subset D\oplus D=A$. Therefore $l=(0,1_K)=(0,1_D)\in {\rm Cay}(K,e)$ satisfies $l^2=e$ and
$l^2=c$ and we obtain $c=e\in K\setminus F$, hence $L=K$ and ${\rm Cay}(K,e)={\rm Cay}(L,c)$.
\\ (ii) Let $F$ have characteristic 2. $K$ is a separable quadratic field extension of $F.$
Hence there is an element
 $X=(w,z)\in A$, $w,z\in D$, such that $X^2+X=f\in F^\times$. As in the proof of Theorem 8, resp. 9,
 this implies $X=(w,0)$ and $w^2+w=f$, $w\in D$, so $ K\subset D$
and $L\oplus L\subset D\oplus D=A$. Therefore $l=(0,1_K)\in {\rm Cay}(K,e)$ satisfies $l^2=e$,
$l^2=c$, and again $c=e\in K\setminus F$. If $c$ lies in a purely inseparable quadratic field extension, this is a contradiction.
Thus in this case $A$ has no nonassociative quaternion division subalgebras.
The rest of the assertion follows analogously as in (i) and the proofs of Theorems 8, 9.
\end{proof}

\begin{theorem} The Cayley-Dickson doublings ${\rm Cay}(D,c)$, ${\rm Cay}_m(D,c)$ and ${\rm Cay}_r(D,c)$
 of a quaternion division algebra $D$ are division algebras for any choice of $c\in D^\times$
not in $F$.
\end{theorem}

\begin{proof} We show that $A={\rm Cay}(D,c)$ has no zero divisors: suppose
$$(0,0)=(u,v)(u',v')=(uu'+c \sigma(v')v,v'u+v\sigma(u'))$$
for $u,u',v,v'\in D$. This is equivalent to
$$uu'+c \sigma(v')v=0\text{ and } v'u+v\sigma(u')=0.$$
Assume $v=0$, then $uu'=0$ and $v'u=0$. Now if also $u=0$ then $(u,v)=(0,0)$, so let $u\not=0$. Then
$u'=v'=0$ and $(u',v')=(0,0).$ So assume $v\not=0$. Since $D$ is a division algebra, $v$ is invertible and the second equation
yields $\sigma(u')=-v^{-1}v'u$, hence
$$u'=-\sigma(u)\sigma(v')\sigma(v^{-1})=-\frac{1}{N_{D/F}(v)}\sigma(u)\sigma(v')v.$$
Substituting this into the first equation gives
$$-\frac{1}{N_{D/F}(v)}u\sigma(u)\sigma(v')v+c\sigma(v')v=(c\sigma(v')-\frac{N_{D/F}(u)}{N_{D/F}(v)}\sigma(v'))v=0.$$
 Since $v\in D^\times$ it follows that
$c\sigma(v')-\frac{N_{D/F}(u)}{N_{D/F}(v)}\sigma(v')=(c-N_{D/F}(\frac{u}{v}))\sigma(v')=0$. $D$
 is a division algebra, so either
$c=N_{D/F}(\frac{u}{v})$ or $v'=0$. If $v'=0$ then $v\sigma(u')=0$ implies $u'=0$ and $(u',v')=(0,0).$ If $v'\not=0$ then
$N_{D/F}(\frac{u}{v})=c\in K\setminus F$ which is a contradiction as $N_{D/F}(\frac{u}{v})\in F$.
\\\\ We next show that $A={\rm Cay}_m(D,c)$ has no zero divisors: suppose
$$(0,0)=(u,v)(u',v')=(uu'+ \sigma(v')cv,v'u+v\sigma(u'))$$
for $u,u',v,v'\in D$. This is equivalent to
$$uu'+ \sigma(v')cv=0\text{ and } v'u+v\sigma(u')=0.$$
Assume $v=0$, then $uu'=0$ and $v'u=0$. Now if also $u=0$ then $(u,v)=(0,0)$, so let $u\not=0$. Then
$u'=v'=0$ and $(u',v')=(0,0).$ So assume $v\not=0$. Since $D$ is a division algebra, $v$ is invertible and the second equation
yields $\sigma(u')=-v^{-1}v'u$, hence
$$u'=-\sigma(u)\sigma(v')\sigma(v^{-1})=-\frac{1}{N_{D/F}(v)}\sigma(u)\sigma(v')v.$$
Substituting this into the first equation gives
$$-\frac{1}{N_{D/F}(v)}N_{D/F}(u)\sigma(v')v+\sigma(v')cv=\sigma(v')(c-\frac{N_{D/F}(u)}{N_{D/F}(v)}))v=0.$$
 Since $v\in D^\times$ it follows that
$\sigma(v')c-\frac{N_{D/F}(u)}{N_{D/F}(v)}\sigma(v')=\sigma(v')(c-N_{D/F}(\frac{u}{v}))=0$. $D$
 is a division algebra, so either
$c=N_{D/F}(\frac{u}{v})$ or $v'=0$. If $v'=0$ then $v\sigma(u')=0$ implies $u'=0$ and $(u',v')=(0,0).$ If $v'\not=0$ then
$N_{D/F}(\frac{u}{v})=c\in K\setminus F$ which is a contradiction as $N_{D/F}(\frac{u}{v})\in F$.

The proof that $A={\rm Cay}_r(D,c)$ has no zero divisors is analogous to the one for $A={\rm Cay}(D,c)$.
\end{proof}

\section{Isomorphisms}

\subsection{} Let $B$ and $D$ be two quaternion algebras over $F$ and $g:D\to B$ an algebra isomorphism.
Let $m\in F^\times$. For $u,v\in D$,
$$G:D\oplus D\to D\oplus D, \quad G(u,v)=(g(u),m^{-1} g(v)) $$
defines the following algebra isomorphisms:
$${\rm Cay}(D,c)\cong {\rm Cay}(B, m^2 g(c)),$$
$${\rm Cay}_m(D,c)\cong {\rm Cay}_m(B, m^2 g(c)),  $$
and
$${\rm Cay}_r(D,c)\cong {\rm Cay}_r(B, m^2g(c)) .$$

\begin{example} For ${\rm char}\,F\not=2$ and $c=c_0+c_1i+c_2j+c_3k\in D^\times\setminus F$,
$${\rm Cay}((a,b)_F,c)\cong {\rm Cay}((e^2a,f^2b)_F,c_0+ec_1i+fc_2j+efc_3k),$$
since $(a,b)_F\cong (e^2a,f^2b)_F$ via $g(i)=ei$, $g(j)=fj$ for $e,f\in F^\times$
and
$${\rm Cay}((a,b)_F,c)\cong {\rm Cay}((b,a)_F,c_0+abc_2i+abc_1j+a^2b^2c_3k),$$
since $(a,b)_F\cong (b,a)_F$ via $g(i)=abj$, $g(j)=abi$.
 Analogously,
$${\rm Cay}_m((a,b)_F,c)\cong {\rm Cay}_m((e^2a,f^2b)_F,c_0+ec_1i+fc_2j+efc_3k),$$
$${\rm Cay}_r((a,b)_F,c)\cong {\rm Cay}_r((e^2a,f^2b)_F,c_0+ec_1i+fc_2j+efc_3k),$$
for all $e,f\in F^\times$ and
$${\rm Cay}_m((a,b)_F,c)\cong {\rm Cay}_m((b,a)_F,c_0+abc_2i+abc_1j+a^2b^2c_3k),$$
$${\rm Cay}_r((a,b)_F,c)\cong {\rm Cay}_r((b,a)_F,c_0+abc_2i+abc_1j+a^2b^2c_3k).$$
For ${\rm char}\,F=2$, $D={\rm Cay}(L,b)=[a,b)$ for some separable quadratic field extension
$L/F$, $L=F(\sqrt{a})$, and $c=m+nj\in D={\rm Cay}(L,b)$ invertible,  $m,n\in L$,
$${\rm Cay}(D,m+nj)\cong {\rm Cay}(D,m+fnj),$$
since $D={\rm Cay}(L,b)\cong {\rm Cay}(L,f^2b)$ via $g(w,z)=(w,fz)$ for $f\in F^\times$.
Analogously,
$${\rm Cay}_m(D,m+nj)\cong {\rm Cay}_m(D,m+fnj),$$
$${\rm Cay}_r(D,m+nj)\cong {\rm Cay}_r(D,m+fnj)$$
for every $f\in F^\times$.
\end{example}

\begin{theorem} Let $B$ and $D$ be two quaternion division algebras over $F$,
$c\in D^\times\setminus F$ and $d\in B^\times\setminus F$, and suppose one of the following:
\\ (i)  ${\rm Cay}(D,c)\cong {\rm Cay}(B,d)$.
\\ (ii)  ${\rm Cay}_m(D,c)\cong {\rm Cay}_m(B,d)$.
\\ (ii) ${\rm Cay}_r(D,c)\cong {\rm Cay}_r(B,d)$.
\\
Then the quaternion algebras $D$ and $B$ are isomorphic and
 either both $c$ and $d$ lie in a separable quadratic field extension $L$ or they both lie in
a purely inseparable quadratic field extension $K$ of $F$, with $L$ (resp. $K$) contained in $D$ and $B$.
If $c,d\in L$ then either  $c=x\bar x d$ or $\bar c=x\bar x d$ for some $x\in L^\times$.
If $c,d\in K$ then ${\rm Cay}(K,c)\cong {\rm Cay}(K,d)$.
\end{theorem}

\begin{proof}  Any isomorphism maps a quaternion subalgebra of ${\rm Cay}(D,c)$ to a quaternion subalgebra of
${\rm Cay}(B,d)$, hence $D\cong B$ in (i) by Theorem 8. The same argument applies for (ii) and (iii), employing Theorem 9.

By Theorem 10, both $c\in D^\times$ and $d\in B^\times$ either both lie in a separable quadratic field extension
contained in $D$ and $B$ (in which case both algebras have a nonassociative quaternion subalgebra) or
both lie in a purely inseparable quadratic field extension
contained in $D$ and $B$ (in which case both algebras have no nonassociative quaternion subalgebra).
\\ Let us look at (i). Suppose both $c\in D^\times$ and $d\in B^\times$  lie in a separable quadratic field extension $L\subset D$ resp. $L'
\subset B$. Then the isomorphism implies that ${\rm Cay}(L,c)\cong {\rm Cay}(L',d)$ (Theorem 10) and thus $L=L'$ and the assertion follows
applying [W, Theorem 2].
If  both $c\in D^\times$ and $d\in B^\times$  lie in a purely inseparable quadratic field extension $K\subset D$ resp. $K'
\subset B$ then the isomorphism implies that ${\rm Cay}(K,c)\cong {\rm Cay}(K',d)$ by Theorem 10.
Since isomorphic algebras have the same nucleus, this implies $K=K'$. The proof is the same for the
other two cases.
\end{proof}

Let $g:D\to D$ be an algebra isomorphism and $m\in F^\times$. By the Theorem of Skolem-Noether [KMRT, (1.4), p.~4],
$g(u)=aua^{-1}$ for a suitable $a\in D^\times$, therefore for $u,v\in D$, the map
$$(u,v)\to (aua^{-1},m^{-1}ava^{-1}) $$
 induces the following algebra isomorphisms:
$${\rm Cay}(D,c)\cong {\rm Cay}(D,m^{2}aca^{-1}),  $$
$${\rm Cay}_m(D,c)\cong {\rm Cay}_m(D,m^{2}aca^{-1}), $$
$${\rm Cay}_r(D,c)\cong {\rm Cay}_r(D,m^{2}aca^{-1}). $$

\begin{theorem} Let  $G:{\rm Cay}(D,c)\to {\rm Cay}(D,d)$ be an algebra isomorphism of two Dickson division algebras
and suppose that both $c$ and $d$ lie in a separable quadratic field extension $L\subset D$.
Write $D={\rm Cay}(L,h)=L\oplus LJ$ with $J^2=h$.
Then $dN_{L/F}(z)=c$ for some $z\in L^\times$ and either there is an invertible $a\in L$ or
$dN_{L/F}(z)=\bar c$ for some $z\in L^\times$ and there is an invertible $a\in LJ$
 such that
$$G((u,v))=(aua^{-1},zava^{-1}).$$
 All $a, z\in L^\times$ yield an isomorphism
$G:{\rm Cay}(D,c)\to {\rm Cay}(D,N_{L/F}(z^{-1})c)$ and all invertible $a\in LJ$, $z\in L^\times$ yield an isomorphism
$G:{\rm Cay}(D,c)\to {\rm Cay}(D,N_{L/F}(z^{-1})\bar c)$.
\end{theorem}

\begin{proof} Restricting $G$ to the nonassociative quaternion subalgebra ${\rm Cay}(L,c)$ yields an isomorphism
$G_0:{\rm Cay}(L,c)\to {\rm Cay}(L,d)$ (Theorem 10). By [W, Thm. 2], either $G_0(u,v)=(u,vz)$ and $c=z\bar z d$ or
$G_0(u,v)=(\bar u,\bar v z)$ and $\bar c=z\bar z d$ for some $z\in L^\times$, $u,v\in L^\times$. In particular,
$G(l)=G((0,1_D))=G_0((0,1_L))=(0,z)=(z,0)(0,1_D)=zl'$, hence the subspace $Dl$ of ${\rm Cay}(D,c)=D\oplus Dl$ is mapped to the
subspace $Dl'$ of ${\rm Cay}(D,d)=D\oplus Dl'$. Then $G(c)=G(l^2)=G(l)G(l)=(zl')(zl')=
(0,z)(0,z)=(d\sigma(z)z,0)$. Since $G$ restricted to $D$ is an automorphism of $D$, there exists an element $a\in D^\times$
such that $G(u)=aua^{-1}$ for all $u\in D$. Thus $G(c)=aca^{-1}$ and $d\sigma(z)z=aca^{-1}$ which means
$$dN_{L/F}(z)=aca^{-1}.$$
Since $G$ is multiplicative we have
$$G((0,v))= G((v,0)(0,1))=G((v,0))G((0,1))=(ava^{-1})(0,z)=(0,zava^{-1} ),$$
 so that $G((u,v))=(aua^{-1},z ava^{-1})$.
 Moreover, for all $u,v,u',v'\in D$,
 $$G((u,v)(u',v'))=G(u,v)G(u',v')$$
  is equivalent to
  \begin{enumerate}
\item $ac\sigma(v')va^{-1}=d \sigma(a^{-1})\sigma(v')va^{-1} N(a)N(z),$
and
\item $zav\sigma(u')a^{-1}=zav\sigma(u')\sigma(a) N(a)^{-1}$
\end{enumerate}
for all $u,u',v,v'\in D$.
Using that $dN_{D/F}(z)=aca^{-1}$, (1) and (2) are true for all $u,u',v,v'\in D$.

Now $G((u,v))=(aua^{-1},zava^{-1})$
restricted to ${\rm Cay}(L,c)$ yields the isomorphism $G_0(u,v)=(u,zv)$  or
$G_0(u,v)=(\bar u,\bar v z)$ for all $u,v\in L$. This implies either
$aua^{-1}=u$ for all $u\in L$ or $aua^{-1}=\bar u$ for all $u\in L$. Thus either $au=ua$ for all $u\in L$,
 or $au=\bar ua$ for all $u\in L$.
\\
We show that $au=ua$ for all $u\in L$ implies $a\in L$: We have $D={\rm Cay}(L,h)$, $a=n+kJ$, $n,k\in L$ $J^2=h$, $mJ=J\bar m$ for all
$m\in L$. $au=nu+kJu=nu+k\bar uJ$ and $ua=un+ukJ$ so that $au=ua$ for all $u\in L$ is equivalent to
$k\bar u= uk$ for all $u\in L$, i.e. to $k=0$ (plug in any $u$ such that $\bar u=-u$, get $k=-k$), i.e to $a\in L$.
\\
And $au=\bar ua$ for all $u\in L$ implies $a\in LJ$:
$au=nu+kJu=nu+k\bar uJ$ and $\bar ua=\bar un+\bar ukJ$ so that
$au=\bar ua$ for all $u\in L$  is equivalent to $n u=\bar u n$ for all $u\in L$, i.e. to $n=0$.
The rest of the assertion is now obvious.
\end{proof}

A similar proof can be used to show:

\begin{theorem}
(a) Let  $G:{\rm Cay}_m(D,c)\to {\rm Cay}_m(D,d)$ be an algebra isomorphism of two Dickson division algebras
and suppose that both $c$ and $d$ lie in a separable quadratic field extension $L\subset D$.
Then $dN_{L/F}(z)=c$ for some $ z\in L^\times$ and
$$G((u,v))=(aua^{-1},zava^{-1})$$
for some $a\in L^\times$.
\\ All $a, z\in L^\times$ yield an isomorphism
$G:{\rm Cay}(D,c)\to {\rm Cay}(D,N_{L/F}(z^{-1})c)$.
\\ (b) Let  $G:{\rm Cay}_r(D,c)\to {\rm Cay}_r(D,d)$ be an algebra isomorphism of two Dickson division algebras
and suppose that both $c$ and $d$ lie in a separable quadratic field extension $L\subset D$.
Write $D={\rm Cay}(L,h)=L\oplus LJ$ with $J^2=h$. Then $dN_{L/F}(z)=c$ for some $a,z\in L^\times$ or
$dN_{L/F}(z)=\bar c$ for some $z\in L^\times$ and invertible $a\in LJ$
 and
$$G((u,v))=(aua^{-1},zava^{-1}).$$
 All $a, z\in L^\times$ yield an isomorphism
$G:{\rm Cay}(D,c)\to {\rm Cay}(D,N_{L/F}(z^{-1})c)$ and all invertible $a\in LJ$, $z\in L^\times$ yield an isomorphism
$G:{\rm Cay}(D,c)\to {\rm Cay}(D,N_{L/F}(z^{-1})\bar c)$.
\end{theorem}

\begin{corollary}
Let $D$ be a division algebra over $F$ and $A = \Cay (D,c)$ or $A = \Cay_r(D,c)$.
Suppose that $c$ lies in a separable quadratic field extension $L\subset D$ of $F$.
Let $\Psi: {\rm Aut}(A)\to {\rm Aut}(D),$ $ F\mapsto F|_D$.
\\ (i) The maps $G:A\to A$, $G((u, v)) = (aua^{-1},z ava^{-1} )$ for all $a\in L^\times$ and all
$z\in L^\times$ with $N_{L/F}(z)=1$ are the only isomorphisms
of $A$ unless $c=-\bar c$. In that case also the maps
 $G((u, v)) = (aua^{-1},z ava^{-1} )$  are isomorphisms of $A$ for all invertible $a\in LJ$
 and all $z\in L^\times$ with $N_{L/F}(z)=-1$.
\\ (ii) ${\rm Aut}(A)/{\rm ker} (\Psi)$ is a
normal subgroup of ${\rm Aut}(D)$.
If $\bar c\not=-c$, then $G\in {\rm ker} (\Psi)$ iff $G((u,v))=(u,N_{L/F}(z)v)$
 with $z\in L^\times$ such that $N_{L/F}(z)=1$.
If $\bar c=-c$, then  $G\in {\rm ker} (\Psi)$ iff $G((u,v))=(u,N_{L/F}(z)v)$
 with $z\in L^\times$ such that $N_{L/F}(z)=\pm 1$.

Hence if $i\in F$ and $\bar c=-c$, also $G_i, G_{-i}\in {\rm ker} (\Psi)$, with $G_{i}((u,v))=(u,iv)$
and $G_{-i}((u,v))=(u,-iv)$.
\\ (iii) The map $\Psi: {\rm Aut}(A)\to {\rm Aut}(D),$ $ F\mapsto F|_D$ is surjective if and only if
every separable quadratic field extension contained in $D$ is isomorphic to $L$.
\end{corollary}

For instance, for all $a\in L^\times$ the maps $G:A\to A$, $G((u, v)) = (aua^{-1}, \pm ava^{-1} )$  are isomorphisms
of $A$ and if $c=-\bar c$ and $i\in F$ also the maps
 $G((u, v)) = (aua^{-1}, \pm i ava^{-1} )$  are isomorphisms of $A$, now for all invertible $a\in LJ$.

\begin{proof} (i) Write $D=\Cay (L,h)=L\oplus LJ$. $G:A\to A$ is an isomorphism if and only if there are elements
 $z\in L^\times$ and $a\in D^\times$ with either $a\in L$ or $a\in LJ$, such that
  $G((u, v)) = (aua^{-1}, zava^{-1})$ and $cN_{L/F}(z)=aca^{-1}$ (Theorems 14 and 15).
If $a\in L$ then $N_{L/F}(z)=1$.

If $a\in LJ$, then $cN_{L/F}(z)=\bar c$. Write $c=m+nX\in F(X)=L$, $n\not=0$, then we can conclude that $cN_{L/F}(z)=\bar c$
 is equivalent to $N_{L/F}(z)=-1$ and $m=0.$
\\ (ii) Let $F\in {\rm Aut}(A)$, then $F$ stabilizes $D$, therefore $F|_D\in {\rm Aut}(D)$ and
 $\Psi: F\mapsto F|_D$ is a homomorphism with kernel ${\rm ker} (\Psi)$ which is a normal subgroup.
 Theorems 14, 15 now imply the remaining assertion.
 \\ (iii) Suppose every separable quadratic field extension contained in $D$ is isomorphic to $L$.
For every $f\in {\rm Aut}(D)$ there is an invertible element $a\in D$ such that $f(u)=aua^{-1}$ and there is a separable
quadratic field extension $K$ such that $a\in K$. By assumption, $K=L$, hence we choose
$G((u,v))=(aua^{-1}, ava^{-1})$ to obtain $\Psi(G)=f.$
\\ Suppose now that $D$ contains two non-isomorphic separable quadratic field extensions $L$ and $K$, $c\in L$.
Assume that $\sigma(c)\not=-c$.
Choose $f(u)=bub^{-1}$ with $b\in K\setminus F$. Then we need to find some $a\in L$ such that $G(aua^{-1}, 0)=(bub^{-1},0)$.
I.e. $aua^{-1}=bub^{-1}$ for all $u\in D$. There is no such $a\in L$: in particular for all $u\in L$ this would mean that
$u=bub^{-1}$, thus $bu=ub$ for all $u\in L$. This implies $b\in L$, a contradiction. Here $\Psi$ is not
surjective.
\\ If $c=-\sigma( c)$, we also have to check the possibility that $aua^{-1}=\sigma (u) $ for all $u\in L$, $a\in LJ$.
Choose $f(u)=bub^{-1}$ with $b\in K\setminus F$. Then we need to find some $a\in LJ$ such that $G(aua^{-1}, 0)=(bub^{-1},0)$.
I.e., $aua^{-1}=bub^{-1}$ for all $u\in D$. In particular, for all $u\in L$ this would mean that
$\sigma(u)=bub^{-1}$, thus $bu=\sigma(u)b$ for all $u\in L$. This implies
$b\in LJ$, a contradiction. Here $\Psi$ is not surjective
 \end{proof}

An analogous proof shows:

\begin{corollary}
Let $D$ be a division algebra over $F$ and $A = \Cay_m(D,c)$.
Suppose that $c$ lies in a separable quadratic field extension $L\subset D$ of $F$.
Let $\Psi: {\rm Aut}(A)\to {\rm Aut}(D),$ $ F\mapsto F|_D$.
\\ (i) The maps $G:A\to A$, $G((u, v)) = (aua^{-1},z ava^{-1} )$ for all $a\in L^\times$ and all
$z\in L^\times$ with $N_{L/F}(z)=1$ are the isomorphisms
of $A$.
\\ (ii) ${\rm Aut}(A)/{\rm ker} (\Psi)$ is a
normal subgroup of ${\rm Aut}(D)$ and $G\in {\rm ker} (\Psi)$ iff $G((u,v))=(u,N_{L/F}(z)v)$
 with $z\in L^\times$ such that $N_{L/F}(z)= 1$.
\\ (iii) The map $\Psi: {\rm Aut}(A)\to {\rm Aut}(D),$ $ F\mapsto F|_D$ is surjective if and only if
every separable quadratic field extension contained in $D$ is isomorphic to $L$.
\end{corollary}

\subsection{} We now address the question if there are isomorphisms  between the three different classes of
Cayley-Dickson doublings.

\begin{proposition} Let $B$, $D$ be two quaternion division algebras, $c,d\in D^\times\setminus F$.
\\ (i) If ${\rm Cay}(D,c)\cong {\rm Cay}_r(B,d)$ then $B\cong D$.
\\ (ii) If ${\rm Cay}_m(D,c)\cong {\rm Cay}(B,d)$ then $B\cong D$.
\\ (iii) If ${\rm Cay}_m(D,c)\cong {\rm Cay}_r(B,d)$ then $B\cong D$.
\\  Moreover, the isomorphisms imply either both $c$ and $d$ lie in a separable quadratic field extension $L$ or they both lie in
a purely inseparable quadratic field extension $K$ of $F$, with $L$ (resp. $K$) being contained in $D$ and $B$.
If $c,d\in L$ then either  $c=x\bar x d$ or $\bar c=x\bar x d$ for some $x\in L^\times$.
If $c,d\in K$ then ${\rm Cay}(K,c)\cong {\rm Cay}(K,d)$.
\end{proposition}

This follows directly from Theorems 8, 9 and the proof of Theorem 13.

\begin{theorem}  Let $D$ be a quaternion division algebra and $c,d\in L\setminus F$ with $L\subset D$ a separable quadratic extension.
Then:
\\ (i) ${\rm Cay}(D,c)\not\cong {\rm Cay}_r(D,d)$,
\\ (ii) ${\rm Cay}_m(D,c) \not\cong  {\rm Cay}(D,d)$ and
\\ (iii) ${\rm Cay}_m(D,c) \not\cong  {\rm Cay}_r(D,d).$
\end{theorem}

\begin{proof} (i)
Suppose there is an algebra isomorphism $G:{\rm Cay}(D,c)=D\oplus Dl\to {\rm Cay}_r(D,d)=D\oplus Dl'$.
Restricting $G$ to the nonassociative quaternion subalgebra ${\rm Cay}(L,c)$ yields an isomorphism
$G_0:{\rm Cay}(L,c)\to {\rm Cay}(L,d)$ by Theorem 10. By [W, Thm. 2], either $G_0(u,v)=(u,vz)$ and $c=z\bar z d$ or
$G_0(u,v)=(\bar u,\bar v z)$ and $\bar c=z\bar z d$ for some $z\in L^\times$, $u,v\in L^\times$. In particular,
$G(l)=G((0,1_D))=G_0((0,1_L))=(0,z)=(z,0)(0,1_D)=zl'$,
i.e., $G(l)=zl'$ for some invertible $z\in L$. Then $G(c)=G(l^2)=G(l)G(l)=(zl')(zl')=
(0,z)(0,z)=(N_{D/F}(z)d,0)$. Since $G$ restricted to $D$ is an automorphism of $D$, there exists an element $a\in D^\times$
such that $G(u)=aua^{-1}$ for all $u\in D$. Thus $G(c)=aca^{-1}$ and $N_{D/F}(z)d=aca^{-1}$.
Since $G$ is multiplicative we have
$$G((0,v))= G((v,0)(0,1))=G((v,0))G((0,1))=(ava^{-1})(0,z)=(0,zava^{-1} ),$$
 so that $G((u,v))=(aua^{-1},z ava^{-1})$.
 Moreover, we must have
 $$G((u,v)(u',v'))=G(u,v)G(u',v')$$
  which is equivalent to
  \begin{enumerate}
  \item $ac\sigma(v')va^{-1}=d \sigma(a^{-1})\sigma(v')va^{-1} N(a)N(z),$
\item $zav\sigma(u')a^{-1}=zav\sigma(u')\sigma(a) N(a)^{-1}$
\end{enumerate}
for all $u,u',v,v'\in D$.
Now (2) is the same equation as (2) in Theorem 14, thus true for all $z\in L^\times.$
(1) is equivalent to
$$c\sigma(v')v=\sigma(v')vaca^{-1} $$
for all $v,v'\in D$, which implies that $ca=ac$.
Write $D={\rm Cay}(L,h)=L\oplus LJ$ with $J^2=h$.
Since both ${\rm Cay}(D,c)$ and ${\rm Cay}_r(D,c)$ contain the unique nonassociative quaternion subalgebra
${\rm Cay}(L,c)$, respectively ${\rm Cay}(L,d)$, the same argument as used
in the proofs of Theorems 14, 15 implies that restricting $G$ to ${\rm Cay}(L,c)$ yields that
either $a\in L$ or $a\in LJ$. The additional requirement that $ca=ac$ forces $a\in L$ and so equation (1)
is the same as $cv=vc$ for all $v\in D$, implying the contradiction that $c\in F^\times$.
\\ (ii) Suppose there is such an algebra isomorphism  $G:{\rm Cay}(D,c)=D\oplus Dl\to {\rm Cay}_m(D,d)=D\oplus Dl'$.
Restricting $G$ to the nonassociative quaternion subalgebra ${\rm Cay}(L,c)$ again yields an isomorphism
$G_0:{\rm Cay}(L,c)\to {\rm Cay}(L,d)$ (Theorem 10). As in (i), we conclude that
$G(l)=G((0,1_D))=G_0((0,1_L))=(0,z)=(z,0)(0,1_D)=zl'$,
i.e., $G(l)=zl'$ for some invertible $z\in L$. Thus $G(c)=G(l^2)=G(l)G(l)=(zl')(zl')=
(0,z)(0,z)=(\bar zdz,0)$. Since $G$ restricted to $D$ is an automorphism of $D$, there exists an element $a\in D^\times$
such that $G(u)=aua^{-1}$ for all $u\in D$. Thus $G(c)=aca^{-1}$ and $\sigma( z)dz=aca^{-1}$.
Moreover, $G((u,v))=(aua^{-1},zava^{-1})$. Since $G$ is multiplicative we must have
 $$G((u,v)(u',v'))=G(u,v)G(u',v')$$
  which is equivalent to
  \begin{enumerate}
   \item $a\sigma(v')cva^{-1}=d \sigma(a^{-1})\sigma(v')va^{-1} N(a)N(z),$
\item $zav\sigma(u')a^{-1}=zav\sigma(u')\sigma(a) N(a)^{-1}$
\end{enumerate}
for all $u,u',v,v'\in D$.
Now (2) is  the same equation as (2) in (i), thus true for all $z\in L^\times$ and (1) is equivalent to
$$\sigma( v')cv=c\sigma( v')v$$
for all $v,v'\in D$, which implies that $cw=wc$ for all $w\in D$, i.e. $c\in F^\times$, a contradiction.
\\ (iii) is proved analogously.
\end{proof}

\section{Derivations}

Let $F$ be a field of characteristic not 2 and $B$ a quaternion division algebra.
 Let $A={\rm Cay}(B,c)$,
$A={\rm Cay}_m(B,c)$ or $A={\rm Cay}_r(B,c)$.
 Let $D\in {\rm Der}(A)$, i.e. $D$ is an $F$-linear map such that $D(xy)=D(x)y+xD(y)$.
Obviously, $D(1)=0$. Assume $D((c,0))=(p,q)$ and $D(l)=(r,s)$ with $p,q,r,s\in D$.
Then
\begin{enumerate}
\item $D(c)=D(l^2)=D(l)l+lD(l)$,
\item $D((cl)c)=D(c(lc))=D(N_{D/F}(c)l)=N_{D/F}(c)D(l)=(c\bar cr,c\bar cs)$,
\item $D(cl)l+clD(l)=D((cl)l)=D((cl)l)=D(c^2)=D((c+\bar c)c-c\bar c)=(c+\bar c) D(c).$
\end{enumerate}

We obtain the following partial results:

\begin{proposition} Let $A={\rm Cay}(B,c)$.\\
 (i)  $D((c,0))=(0,0)$ and $D(l)=(r,s)$ for $r,s\in B$ such that
$\bar s=-s$, $\bar r=-r$ and $\overline{rc}=-rc$.
\\ (ii) The $F$-linear map
$$D_0((u,v))=(0, sv)$$
is a derivation of $A$ if and only if $\bar s=-s$.
\\ (iii) Let $K$ be a separable quadratic subfield of $B$ such that $c\in K$.
Then
 ${\rm Der}(A_0)\hookrightarrow {\rm Der}(A)$
for the nonassociative quaternion subalgebra $A_0={\rm Cay}(K,c)$.
\end{proposition}

\begin{proof}
(i) Equation (1) implies that $(p,q)=(c(s+\bar s), r+\bar r),$
therefore
$$p=c(s+\bar s),\quad q=r+\bar r.$$
Moreover, by equation (2),
$$(cqc+crc+c\bar q c,p\bar c+ sc\bar c+c\bar p) =(c\bar cr,c\bar cs).$$
 Hence
$$cqc+crc+c\bar q c=c\bar cr \text{ and } p\bar c+ sc\bar c+c\bar p=c\bar cs.$$
Since $p=c(s+\bar s)$ and $ q=r+\bar r$ this implies that $2c(s+\bar s)\bar c=0$ and
$3rc+ 2\bar rc=\bar cr$, thus $s+\bar s=0$ and $p=0$. Besides, (3) yields
$$(csc+c\bar s c,cq+cr+c\bar r) =(0,(c+\bar c)q).$$
We obtain $cq+cr+c\bar r=(c+\bar c)q$ and substituting $q=r+\bar r$ yields
$2cq=(c+\bar c)q$, thus $(c-\bar c)q =0$ and $q=0$. Therefore
$$D((c,0))=(0,0).$$
Since $q=0$  also $r+\bar r=0$. Substituting $\bar r=-r$ into the equation $3rc+ 2\bar rc=\bar cr$
yields $rc=\bar cr,$ hence also $rc=-\overline{rc}$.
\\ (ii) is an easy calculation.
\\ (iii) We know that $\delta\in {\rm Der}(A_0)$ iff $\delta=\delta_s(u',v')=(0,sv')$ for
 some $s\in K$ such that $s+\bar s=0$, with $u',v'\in K$. The map
 $$F:{\rm Der}(A_0)\longrightarrow {\rm Der}(A), \delta_s\to D_s $$
 with $D_s(u,v)=(0,sv)$ for $u,v\in B$
 is an injective algebra homomorphism.
\end{proof}

\begin{proposition} Let $A={\rm Cay}_m(B,c)$ and $D\in {\rm Der}(A)$. \\
(i) $D((c,0))=(cs+\bar sc,0)$ and
$D(l)=(r,s)$ for $r,s\in D$ such that
$\bar s=-s$, $\bar r=-r$ and $rc=-\overline{rc}.$
\\ (ii) For $s\in B$, the $F$-linear map
$$D_0((u,v))=(0, sv)$$
is a derivation of $A$ if and only if   $cs+\bar sc=0$.
\\ (iii) Let $K$ be a separable quadratic subfield of $B$ such that $c\in K$.
Then ${\rm Der}(A_0)\hookrightarrow {\rm Der}(A)$
for the nonassociative quaternion subalgebra $A_0={\rm Cay}(K,c)$.
\end{proposition}

\begin{proof}
(i) By equation (1), $(p,q)=(cs+\bar s c, r+\bar r)$, therefore
$$p=cs+\bar sc,\quad q=r+\bar r.$$
Moreover, by equation (2),
$$(cqc+crc+\bar q c^2,p\bar c+ sc\bar c+c\bar p)=(c\bar cr,c\bar cs).$$
 Hence
$$cqc+crc+\bar q c^2=c\bar cr \text{ and } p\bar c+ sc\bar c+c\bar p=c\bar cs.$$
Since $p=cs+\bar s c$ and $ q=r+\bar r$ this implies that
$0=\bar c(s+\bar c)$, i.e. $s+\bar s=0$ and $p=cs-sc$, as well as
$3rc+ 2\bar rc=\bar cr$. Besides, using equation (3) implies that
$$(cp+csc+\bar s c^2,cq+cr+c\bar r)=((c+\bar c)p,(c+\bar c)q).$$
We obtain $cq+cr+c\bar r=(c+\bar c)q$ and substituting $q=r+\bar r$ yields
$2cq=(c+\bar c)q$, thus $(c-\bar c)q =0$ and $q=0$.
Moreover, we get $cp+csc+\bar sc=(c+\bar c)p$ which yields
$(c+\bar c)s=s(c+\bar c) $ for $\bar s=-s$ which is always true.
Therefore
$$D((c,0))=(cs-sc,0).$$
Since $q=0$  also $r+\bar r=0$. Substituting $\bar r=-r$ into the equation $3rc+ 2\bar rc=\bar cr$
yields $rc=\bar cr,$ hence also $rc=-\overline{rc}$.
\\ (ii) is an easy calculation  and (iii) is proved as in Proposition 20.
\end{proof}

\begin{proposition} Let $A={\rm Cay}_r(B,c)$ and $D\in {\rm Der}(A)$. \\
(i) $D((c,0))=(0,0)$ and
$D(l)=(r,s)$ for $r,s\in D$ such that
$\bar s=-s$, $\bar r=-r$ and $\overline{rc}=-rc$.
\\ (ii) For $s\in B$, the $F$-linear map
$$D_0((u,v))=(0, sv)$$
is a derivation of $A$ if and only if   $s+\bar s=0$.
\\ (iii) Let $K$ be a separable quadratic subfield of $B$ such that $c\in K$. Then
 ${\rm Der}(A_0)\hookrightarrow {\rm Der}(A)$ for the nonassociative quaternion subalgebra $A_0={\rm Cay}(K,c)$.
\end{proposition}

\begin{proof}
(i)  By equation (1), $(p,q)=(sc+\bar s c, r+\bar r)$, therefore
$$p=(s+\bar s)c,\quad q=r+\bar r.$$
Moreover, by equation (2),
$$(qc^2+crc+\bar q c^2,p\bar c+ sc\bar c+c\bar p)=(c\bar cr,c\bar cs).$$
 Hence
$$qc^2+crc+\bar q c^2=c\bar cr \text{ and } p\bar c+ sc\bar c+c\bar p=c\bar cs.$$
Since $p=(s+\bar s) c$ and $ q=r+\bar r$ this implies that
 $s+\bar s=0$ and $p=0$, as well as
$3rc+ 2\bar rc=\bar cr$. Besides, using (3) gives
$$(pc+sc^2+\bar s c^2,qc+cr+c\bar r)=((c+\bar c)p,(c+\bar c)q).$$
We obtain $qc+cr+c\bar r=(c+\bar c)q$ and substituting $q=r+\bar r$ yields
$q=0$. Therefore
$$D((c,0))=(0,0).$$
Since $q=0$  also $r+\bar r=0$. Substituting $\bar r=-r$ into the equation $3rc+ 2\bar rc=\bar cr$
yields $rc=\bar cr,$ hence also $rc=-\overline{rc}$.
\\ (ii) again is an easy calculation and (iii) proved as in Proposition 20.
\end{proof}

\begin{remark}
Obviously $r=0$ will work in part (i) of all of the last three propositions. Since $c\in L$ for some quadratic
subfield $L$ of $B$, any $r\in L^\times$  will give a contradiction.
\end{remark}

If $A$ is a real division algebra of dimension 8 then its derivation algebra must be one of  the following
[B-O1]: Compact $G_2$, $su(3)$, $su(2)\oplus su(2)$, $su(2)\oplus N$ with $N$ an abelian ideal of dimension 0 or 1
or $N$ with $N$ abelian and of dimension 0, 1, or 2. From Corollaries 16 and 17 we conclude:

\begin{corollary} For all Dickson algebras $A={\rm Cay}(B,c)$, $A={\rm Cay}_m(B,c)$ and $A={\rm Cay}_r(B,c)$ over
$\mathbb{R}$,
$${\rm Der}(A)\cong su(2)\oplus F.$$
The decomposition of  $A$ into irreducible $su(2)$-modules has the form
$1+1+3+3$ (denoting an irreducible $su(2)$-module by its dimension).
\end{corollary}

\begin{proof} By Corollary 16, repectively 17, (iii), ${\rm dim}_\mathbb{R}{\rm ker}(\Psi)=
{\rm dim}_\mathbb{R}\{z\in \mathbb{C}^\times\,|\, N_{\mathbb{C}/\mathbb{R}}(z)=1\}=1$
 and ${\rm dim}_\mathbb{R}{\rm Aut}(A)={\rm dim}_\mathbb{R}{\rm Aut}(B)+1=4$, therefore
${\dim}_\mathbb{R}{\rm Aut}(A)={\dim}_\mathbb{R}{\rm Der}(A)= 4$ and ${\rm Aut}(A)$ has a normal
subgroup of dimension 1. Since all algebras are unital and we have the unique quaternion subalgebra $\mathbb{H}$
and the unique nonassociative quaternion subalgebra ${\rm Cay}(\mathbb{C},c)$, the rest
of the assertion follows directly from [B-O1, Proposition 6.1].
\end{proof}

Examples of real division algebras with ${\rm Der}(A)\cong su(2)$ where the decomposition of $A$ into irreducible
$su(2)$-modules has the form $1+1+3+3$ were already given by Rochdi [R]. These were  real noncommutative
 Jordan algebras.
 To our knowledge there are no known examples yet of unital algebras $A$ with ${\rm Der}(A)\cong su(2)\oplus F,$
  where the decomposition of $A$ into irreducible $su(2)$-modules has the form $1+1+3+3$. Thus our algebras would answer
  the second part of
  the question posed by Benkart and Osborn [B-O1, p.~292], if there are eight-dimensional real division algebras with derivation algebra
isomorphic to $su(2)\oplus F$ and decomposition of  $A$ into irreducible $su(2)$-modules of the form
$1+1+3+3$.

\section{The opposite algebras}

Let $D$ be a quaternion  algebra over $F$ and $A^{\rm op}$ be the opposite algebra of the Dickson algebra
\begin{enumerate}
\item $A={\rm Cay}(D,c)$,
\item$A={\rm Cay}_m(D,c)$ or
\item $A={\rm Cay}_r(D,c)$.
\end{enumerate}
Then the multiplication in $A^{\rm op}$ is given by
\begin{enumerate}
\item $(u,v)\cdot(u',v')=(u',v')(u,v)=(u'u+c \bar v v',vu'+v'\bar u),$
\item  $(u,v)\cdot(u',v')= (u'u+ \bar v cv',vu'+v'\bar u)$ or
\item $(u,v)\cdot(u',v')=(u'u+ \bar v v'c,vu'+v'\bar u).$
\end{enumerate}
Our previous results easily carry over to the opposite algebras.
$A^{\rm op}$ is again not third power-associative and not quadratic.
If $D$ is a quaternion division algebra over $F$, then ${\rm Nuc}(A^{\rm op})=F$ and
 ${\rm C}(A^{\rm op})=F$ (Lemma 4, 7). $A^{\rm op}$ has the unique quaternion subalgebra $D$
 and if $c\in D^\times$ is contained in a separable quadratic field extension
$L$ of $F$ then ${\rm Cay}(L,c)$ is its only nonassociative quaternion subalgebra.
 If $c\in D^\times$ is instead contained in a purely inseparable quadratic field extension $K$
of $F$, then $A^{\rm op}$ has no nonassociative quaternion subalgebras. In that case, if
${\rm Cay}(K',e)$ is a subalgebra of $A$, where $K'$ is a purely inseparable quadratic field extension of
$F$, then ${\rm Cay}(K',e)={\rm Cay}(K,c)$. In particular,
 $A^{\rm op}$ has the same derivation algebra as the Dickson algebra $A$:
if $F=\mathbb{R}$ then ${\rm Der}(A^{\rm op})\cong su(2)\oplus F$ and
the decomposition of the unital algebra $A^{\rm op}$ into irreducible $su(2)$-modules is again of type
$1+1+3+3$.

Since $D\cong D^{\rm op}$ via the involution $\sigma$, we know that ${\rm Cay}(D,c)\cong {\rm Cay}(D^{\rm op},\sigma(c))$,
${\rm Cay}_m(D,c)\cong {\rm Cay}_m(D^{\rm op},\sigma(c))$ and ${\rm Cay}_r(D,c)\cong {\rm Cay}_r(D^{\rm op},\sigma(c))$.

 Suppose $A^{\rm op}\cong {\rm Cay}(B,d)$, $A^{\rm op}\cong {\rm Cay}_m(B,d)$ or $A^{\rm op}\cong {\rm Cay}_r(B,d)$.
Then $D^{\rm op}\cong B$ since he quaternion subalgebras are unique. However, the opposite algebra of an eight-dimensional Dickson algebra is not a Dickson algebra:

\begin{theorem}  Let $D$ be a quaternion division algebra and $c\in L\setminus F$ with $L\subset D$ a separable quadratic extension.
Then:
\\ (i) ${\rm Cay}(D,c)^{\rm op}\not\cong {\rm Cay}(D,d)$, ${\rm Cay}(D,c)^{\rm op}\not\cong {\rm Cay}_m(D,d)$
and ${\rm Cay}(D,c)^{\rm op}\not\cong {\rm Cay}_r(D,d)$,
\\ (ii) ${\rm Cay}_m(D,c)^{\rm op} \not\cong  {\rm Cay}(D,d)$, ${\rm Cay}_m(D,c)^{\rm op} \not\cong  {\rm Cay}_m(D,d)$ and
${\rm Cay}_m(D,c)^{\rm op} \not\cong  {\rm Cay}_r(D,d)$,
\\ (iii) ${\rm Cay}_r(D,c)^{\rm op} \not\cong  {\rm Cay}(D,d)$, ${\rm Cay}_r(D,c)^{\rm op} \not\cong  {\rm Cay}_m(D,d)$ and
${\rm Cay}_r(D,c)^{\rm op} \not\cong  {\rm Cay}_r(D,d)$
for all $d\in D$.
\end{theorem}

\begin{proof} Let $A={\rm Cay}(D,c)$, $A={\rm Cay}_m(D,c)$ or $A={\rm Cay}_r(D,c)$
 and suppose  there exists an algebra isomorphism
$G:A^{\rm op}\to {\rm Cay}(D,d)$ (respectively, $G:A^{\rm op}\to {\rm Cay}_m(D,d)$ or $G:A^{\rm op}\to {\rm Cay}_r(D,d)$).
Restricting $G$ to the nonassociative quaternion subalgebra ${\rm Cay}(L,c)$ of $A^{\rm op}$ yields an isomorphism
$G_0:{\rm Cay}(L,c)\to {\rm Cay}(L,d)$ and by [W, Thm. 2], either $G_0(u,v)=(u,vz)$ and $c=z\bar z d$ or
$G_0(u,v)=(\bar u,\bar v z)$ and $\bar c=z\bar z d$ for some $z\in L^\times$, for $u,v\in L^\times$. In particular,
$G(l)=G((0,1_D))=G_0((0,1_L))=(0,z)=(z,0)(0,1_D)=zl'$,
i.e., $G(l)=zl'$ for some invertible $z\in L$. Also, $G(c)=G(l^2)=G(l)G(l)=(zl')(zl')=
(0,z)(0,z)=(N_{D/F}(z)d,0)$.
Since $G$ restricted to $D$ is an isomorphism
$$G|_{D^{\rm op}}: D^{\rm op} \to D,$$
 there exists an element $a\in D^\times$
such that $G(u)=a\sigma( u)a^{-1}$ for all $u\in D$. Thus $G(c)=a \sigma( c)a^{-1}$ and $N_{D/F}(z)d=a \sigma( c) a^{-1}$.
Since $G$ is multiplicative we have
$$G((0,v))= G((v,0)(0,1))=G((v,0))G((0,1))=(ava^{-1})(0,z)=(0,za\sigma(v)a^{-1} ),$$
 so that $G((u,v))=(a\sigma(u)a^{-1},z a\sigma(v)a^{-1})$.
 Moreover, we must have
 $$G((u,v)(u',v'))=G(u,v)G(u',v')$$
 for all $u,u',v,v'\in D$.
 \\ (i) Suppose $A={\rm Cay}(D,c)$  and $G:A^{\rm op}\to {\rm Cay}(D,d)$. Then the multiplicativity of $G$ is equivalent to
  \begin{enumerate}
\item $a\sigma(v')v\sigma(c)a^{-1}=d \sigma(a^{-1})v'\sigma(a) a\sigma(v)a^{-1}N_{D/F}(z)$
and
\item $za\sigma(u')\sigma(v)a^{-1}+zau\sigma(v')a^{-1}=za\sigma(v')a^{-1}a\sigma(u)a^{-1}+za\sigma(v)a^{-1}\sigma(a)^{-1}
u'\sigma(a)$
\end{enumerate}
for all $u,u',v,v'\in D$.
Using that $dN_{D/F}(z)=a\sigma(c)a^{-1}$ and that $a$ is invertible, this is the same as
 \begin{enumerate}
\setcounter{enumi}{2}
\item
$\sigma(v')v \bar c= \bar c \sigma(v')\sigma(v)$
and
\item $\sigma(u')\sigma(v)+u\sigma(v')=\sigma(v')\sigma(u)+\sigma(v)u'$
\end{enumerate}
for all $u,u',v,v'\in D$. Put $u=0$ in (4), then
$$\sigma(u')\sigma(v)=\sigma(v)u'$$
for all $u',v\in D$, a contradiction.
 \\ Since equation (2) and consequently also equation (4) remain the same in case we study
 $G:A^{\rm op}\to {\rm Cay}_m(D,d)$ or $G:A^{\rm op}\to {\rm Cay}_r(D,d)$, the same contradiction
 implies the rest of (i). The same argument applies in (ii) and (iii).
\end{proof}

\begin{example}  Let $\mathbb{H}=(-1,-1)_\mathbb{R}$ denote Hamilton's quaternion algebra.
Then for each $c\in \mathbb{C}$ with $c\not\in\mathbb{R}$, ${\rm Cay}(\mathbb{H},c)$, ${\rm Cay}_m(\mathbb{H},c)$,
${\rm Cay}_r(\mathbb{H},c)$, ${\rm Cay}(\mathbb{H},c)^{\rm op}$, ${\rm Cay}_m(\mathbb{H},c)^{\rm op}$ and
${\rm Cay}_r(\mathbb{H},c)^{\rm op}$ are mutually non-isomorphic unital division algebras over $\mathbb{R}$.
 All contain $\mathbb{H}$ and ${\rm Cay}(\mathbb{R},c)$ as the only (nonassociative) quaternion subalgebras.
 For instance, ${\rm Cay}(\mathbb{H},i)\cong {\rm Cay}(\mathbb{H},fi)$, ${\rm Cay}_m(\mathbb{H},i)\cong
{\rm Cay}_m(\mathbb{H},fi)$ and ${\rm Cay}_r(\mathbb{H},i)\cong {\rm Cay}_r(\mathbb{H},fi)$ for all $f\in \mathbb{R}^\times$
are mutually non-isomorphic division algebras over $\mathbb{R}$. For non-real $c,d\in \mathbb{C}^\times$ we have
$$ {\rm Cay}(\mathbb{H},c)\cong {\rm Cay}(\mathbb{H},d), {\rm Cay}_m(\mathbb{H},c)\cong {\rm Cay}_m(\mathbb{H},d)
\text{ or } {\rm Cay}_r(\mathbb{H},c)\cong {\rm Cay}_r(\mathbb{H},d)
$$
$$ \text{ iff } d=z\bar z c   \text{ for some } z\in \mathbb{C} \text{ with }z\bar z=1.$$

\end{example}

\section{Division algebras of dimension 16 and higher}

An involution $\can:A\to A$ on an $F$-algebra $A$ is called {\it scalar} if all  $\bar xx$ are elements of $F1$.

\begin{theorem} Let $C$ be a flexible division algebra over $F$  with scalar involution $\sigma=\can$
(e.g., an octonion division algebra).
 The Cayley-Dickson doublings ${\rm Cay}(C,c)$, ${\rm Cay}_m(C,c)$, ${\rm Cay}_r(C,c)$,
 ${\rm Cay}^l(C,c)$, ${\rm Cay}^l_m(C,c)$ and ${\rm Cay}^r_r(C,c)$
 are division algebras for any choice of $c\in C^\times\setminus F$
 such that $N_{C/F}(c)\not\in N_{C/F}(C^\times)^2$, i.e. in particular for
 any $c\in C^\times\setminus F$ such that $N_{C/F}(c)\not\in F^{\times 2}$.
\end{theorem}

\begin{proof} Let   $N:C\to F$, $N(x)=N_{C/F}(x)=x\bar x=\bar x x$ be the norm of $C$.
We show that $A={\rm Cay}(C,c)$ has no zero divisors: suppose
$$(0,0)=(r,s)(u,v)=(ru+c (\bar v s),vr+s\bar u)$$
for $r,s,u,v\in C$. This is equivalent to
$$ru+c (\bar v s)=0\text{ and } vr+s\bar u=0.$$
Assume $s=0$, then $ru=0$ and $vr=0$. Hence either $r=0$ and so $(r,s)=0$ or $r\not=0$ and
$u=v=0$.

So let $s\not=0$. Then $s\in C^\times$ and $s\bar u=-vr$.
Thus $ u \bar s=-\bar r\bar v$,
$ (u \bar s)\bar s=u \bar s^2=(-\bar r\bar v) \bar s$ (since $A $ is flexible), thus
$$u=-\frac{1}{N(s)^2}((\bar r\bar v) \bar s)s^2$$
plugged into $ru+c (\bar v s)=0 $ implies that
$$\frac{1}{N(s)^2}r(((\bar r\bar v) \bar s)s^2)=c (\bar v s). $$
Applying the norm we obtain that
$$\frac{1}{N(s)^4}N(r)N( v)N( r) N(s)^3=N(c) N(s)N( v). $$
Therefore
$$\frac{1}{N(s)}N(r)N( v)N( r) =N(c) N(s)N( v), $$
i.e.
$$N(r)^2N( v) =N(c) N(s)^2 N( v), $$

If $v=0$ then $ru=0$ and $s\bar u=0$, thus $u=0$ and $(u,v)=0$.
If $v\not=0$ then $N(v)\not=0$ and we get $N(c)=N(\frac{r}{s})^2$, a contradiction to our initial assumption that
$N(c)\not\in N(D^\times)^2$, unless $r=0$. However, if $r=0$ (and $s\not=0$ as assumed above)
 then the initial two equalities give $u=0$ and $v=0$, so that $(u,v)=(0,0)$.

 The other cases are treated similarly, the argument remains the same, since it is unaffected by the placement of
 the scalar $c$ or by the placement of the parentheses.
\end{proof}

\begin{theorem} Let $F$ have characteristic not 2 and let  $C$ be an octonion division algebra,
 $c\in C^\times\setminus F$. Then
$C$ is the only octonion subalgebra of the algebras ${\rm Cay}(C,c)$, respectively ${\rm Cay}_r(C,c)$.
If there is a quaternion algebra $D$ such that ${\rm Cay}(D,d)$ is a subalgebra of ${\rm Cay}(C,c)$, respectively
${\rm Cay}_r(D,d)$ of ${\rm Cay}_r(C,c)$, then $D$ is a subalgebra of $C$ and $d=c.$
\end{theorem}

\begin{proof} The proof for $A={\rm Cay}(C,c)$ resembles the one of Theorem 8:

 Suppose there is an octonion subalgebra $C'={\rm Cay}(F,e,f,g)$ in $A$.
  Then $C'$ has the standard $F$-basis
 $$\{1,X,Y,XY,Z,XZ,YZ,(XY)Z\}$$ with $X^2=e$, $Y^2=f$, $Z^2=g$, and
 $uZ=Z\sigma( u)$ for  $u=X,Y,XY\in (e,f)_F$.

 Hence there is an element  $X=(u,v)\in A$
with $u,v\in C$ such that $X^2=e\in F^\times$ and an element $Y=(w,z)\in A$
with $w,z\in C$ such that $Y^2=f\in F^\times$ and $XY+YX=0$ (coming from the quaternion subalgebra
$(e,f)_F$). The first is equivalent to
$$(1) \quad u^2+cN_{D/F}(v)=e \text{ and   } (2) \quad  vu+ v\sigma(u)=0.$$
 Thus
if $v=0$, then $u^2=e$ and $X=(u,0)\in C$.\\
Now a similar proof as for Theorem 8 (i)
implies the assertion:
if $v\not=0$ then $v$ is invertible and $\sigma(u)=-u$.
This implies $u^2=-N_{C/F}(u)$, thus $e+N_{C/F}(u)=cN_{C/F}(v)$ i.e.
$N_{C/F}(v^{-1})(e+N_{C/F}(u))=c$ which is a contradiction since
the right hand side lies in $C$ and not in $F$, while the left hand side lies in $F$.\\
Analogously, the second equation implies $w^2=f$ and $Y=(w,0)$, $w\in C$. Hence $(0,0)=XY+YX=(uw+wu,0)$ means
$uw+wu=0$ and so the standard basis $1,X=u,Y=w,XY=uw$ for the octonion algebra $(e,f)_F$ lies in $C$ and we obtain
$(e,f)_F\subset C$.

The same observations also prove that $(f,g)_F\subset C$ and that $Z=(x ,0)$ with $x\subset C$.

Moreover, $uZ=Z\bar u$ for  $u=X,Y,XY\in (e,f)_F$ yields the following:
 $XZ=Z\sigma( X)$ is equivalent to $ux=x\sigma( u)$,  $YZ=Z\sigma (Y)$ is equivalent to $vx=x\sigma(v)$ and
  $(XY)Z=Z\bar (XY)$ is equivalent to $(uv)x=x\sigma( uv).$ Thus the standard basis for the octonion algebra
$C'={\rm Cay}(F,e,f,g)$ lies in $C$ and we obtain $C'=C$.
\\ For the proof of the second part of the statement we show that $D\subset C$ as above and conclude
$D\oplus D\subset C\oplus C=A$. Therefore $l=(0,1_D)\in {\rm Cay}(D,d)$ satisfies $l^2=d$ and
$l^2=c$ and again $c=d\in D$.

The proof for the other algebra is analogous.
\end{proof}

\begin{theorem}  Let  $C$ be an octonion division algebra and $c\in C^\times\setminus F$
contained in a separable quadratic field extension
 of $F$. If
${\rm Cay}(L,e)$ is a nonassociative quaternion subalgebra of the algebra $A={\rm Cay}(C,c)$ or $A={\rm Cay}_r(C,c)$
 then $c\in  L\subset C$ and ${\rm Cay}(L,e)={\rm Cay}(L,c)$.
\\ If $c\in C^\times\setminus F$ is contained in a purely inseparable quadratic field extension $K$
of $F$, then $A$ has no nonassociative quaternion subalgebras. If
${\rm Cay}(K',e)$ is a subalgebra of $A$, where $K'$ is a purely inseparable quadratic field extension of
$F$, then ${\rm Cay}(K',e)={\rm Cay}(K,c)$.
\end{theorem}

\begin{proof}  Let $F$ have characteristic not 2 and $ L=F(\sqrt{f})$. Then there is an
element $Y=(w,z)\in A$, $w,z\in C$, such that $Y^2=f\in F^\times$.
Analogously as in the  proof of Theorem 27, this implies $Y=(w,0)$ and $w^2=f$, $w\in C$, so $ L\subset C$
and $L\oplus L\subset C\oplus C=A$. Therefore $l=(0,1_L)=(0,1_C)\in {\rm Cay}(L,e)$ satisfies $l^2=e$ and
$l^2=c$ and we obtain $c=e\in L$.
\\  Let $F$ have characteristic 2. $L$ is a separable quadratic field extension of $F.$
Hence there is an element
 $X=(w,z)\in A$, $w,z\in C$, such that $X^2+X=f\in F^\times$. As in the proof of Theorem 8, resp. 9,
 this implies $X=(w,0)$ and $w^2+w=f$, $w\in C$, so $ L\subset C$
and $L\oplus L\subset C\oplus C=A$. Therefore $l=(0,1_L)\in {\rm Cay}(L,e)$ satisfies $l^2=e$ and
$l^2=c$ and again $c=e\in L$. If $c$ lies in a purely inseparable quadratic field extension, this is a contradiction.
Thus in this case there are no nonassociative quaternion division subalgebras in $A.$
\end{proof}

\begin{example} Let $F=\mathbb{Q}$ and $C={\rm Cay}(\mathbb{Q},a,b,e)$ an octonion algebra.

For all $a,b,e<0$, $C$ is  a division algebra and ${\rm Cay}(C,c)$, ${\rm Cay}_m(C,c)$, ${\rm Cay}_r(D,c)$,
 ${\rm Cay}^l(C,c)$, ${\rm Cay}^l_m(C,c)$ and ${\rm Cay}^r_r(C,c)$
are division algebras for all
$c=x_0+x_1i+x_2j+x_3k+x_4l+x_5il +x_6jl+x_7kl$, such that the positive rational number
$$N_{D/F}(c)=x_0^2-ax_1^2-bx_2^2+abx_3^2-ex_4^2+aex_5^2+bex_6^2-abex_7^2$$
 is not a square. E.g., let $\mathbb{O}={\rm Cay}(\mathbb{Q},-1,-1,-1)$. If $
c= x_0^2+x_1^2+x_2^2+x_3^2+x_4^2+x_5^2+x_6^2+x_7^2$ is not a
 square then ${\rm Cay}(\mathbb{O},c)$, ${\rm Cay}_m(\mathbb{O},c)$, ${\rm Cay}_r(\mathbb{O},c)$,
  ${\rm Cay}^l(\mathbb{O},c)$, ${\rm Cay}^l_m(\mathbb{O},c)$ and ${\rm Cay}^r_r(\mathbb{O},c)$
are division algebras and $\mathbb{O}$ is the only octonion subalgebra of
 ${\rm Cay}(\mathbb{O},c)$ and ${\rm Cay}_r(\mathbb{O},c)$.
\end{example}

\begin{theorem} Let $F$ have characteristic not 2 and let  $C$ and $H$ be octonion division algebras, $c\in C^\times\setminus F$,
$d\in H^\times\setminus F$. If ${\rm Cay}(C,c)\cong {\rm Cay}(H,d)$ or ${\rm Cay}_r(C,c)\cong
{\rm Cay}_r(H,d)$ then the algebras $C$ and $H$ are isomorphic.
 Moreover, if $c$ is contained in a quadratic field extension of $C$ then either both $c$ and $d$
 lie in a separable quadratic field extension $L$ or they both lie in
a purely inseparable quadratic field extension $K$ of $F$, with $L$ (resp. $K$) being contained in $C$ and $H$.
If $c,d\in L$ then either  $c=x\sigma(x) d$ or $\sigma( c)=x\sigma(x) d$ for some $x\in L^\times$.
If $c,d\in K$ then ${\rm Cay}(K,c)\cong {\rm Cay}(K,d)$.
\end{theorem}

\begin{proof} Every isomorphism maps the octonion subalgebra of ${\rm Cay}(C,c)$ to the octonion subalgebra of
${\rm Cay}(H,d)$, hence $C\cong H$ by Theorem 28. The rest of the assertion is proved analogously as Theorem 13
using Theorem 29.
\end{proof}

\begin{theorem} Let $F$ have characteristic not 2 and $C$ be an octonion division algebra.
\\ (i) Let
$G:{\rm Cay}(C,c)\to {\rm Cay}(C,d)$ be an algebra isomorphism
and suppose that both $c$ and $d$ lie in a quadratic field extension $L\subset C$.
Then $d=m^2 g(c)$ for some $g\in {\rm Aut}(C)$, $m\in F^\times$ and
$$G((u,v))=(g(u),g(v)m^{-1}).$$
Every $g\in {\rm Aut}(C)$ and $m\in F^\times$ indeed yield an automorphism
$G:{\rm Cay}(C,c)\to {\rm Cay}(C,m^2 g(c)),.$
\\ (ii) Let
$G:{\rm Cay}_r(C,c)\to {\rm Cay}_r(C,d)$ be an algebra isomorphism
and suppose that both $c$ and $d$ lie in a quadratic field extension $L\subset C$.
Then $d=m^2 g(c)$ for some $g\in {\rm Aut}(C)$, $m\in F^\times$ and
$$G((u,v))=(g(u),g(v)m^{-1}).$$
Every $g\in {\rm Aut}(C)$ and $m\in F^\times$ indeed yield an automorphism
$G:{\rm Cay}(C,c)\to {\rm Cay}(C,m^2 g(c)).$
\end{theorem}

\begin{proof} In both (i) and (ii), restricting $G$ to the nonassociative quaternion subalgebra ${\rm Cay}(L,c)$ yields an isomorphism
$G_0:{\rm Cay}(L,c)\to {\rm Cay}(L,d)$ (Theorem 29). By [W, Thm. 2], either $G_0(u,v)=(u,vz)$ and $c=z\sigma( z) d$ or
$G_0(u,v)=(\sigma( u),\sigma( v) z)$ and $\sigma( c)=z\sigma( z) d$ for some $z\in L^\times$, $u,v\in L^\times$. In particular,
$G(l)=G((0,1_C))=G_0((0,1_L))=(0,z)=(z,0)(0,1_C)=zl'$, hence the subspace $Cl$ of ${\rm Cay}(C,c)=C\oplus Cl$ is mapped to the
subspace $Cl'$ of ${\rm Cay}(C,d)=C\oplus Cl'$.

Then $G(c)=G(l^2)=G(l)G(l)=(zl')(zl')=(0,z)(0,z)=(d\bar z z,0)$. $G$ restricted to $C$ is an automorphism $g$ of $C$.
 Thus $d\bar zz=g(c)$ which means
$$dN_{L/F}(z)=g(c).$$
Since $G$ is multiplicative we have
$$G((0,v))= G((v,0)(0,1))=G((v,0))G((0,1))=(g(v),0)(0,z)=(0,z g(v) ),$$
 so that $G((u,v))=(g(u),z g(v))$. $G$ is multiplicative, thus we have
 $$G((u,v)(u',v'))=G(u,v)G(u',v')$$
 for all $u,u',v,v'\in C.$
 In (i), this is
 equivalent to
  \begin{enumerate}
\item $g(c)(g(\overline{v'})g(v))=d((\bar z \overline{g(v')})(g(v)z)),$
and
\item $(g(v')g(u))z+(g(v)g(\overline{u'}))z=(g(v')z)g(u)+(g(v)z)\overline{g(u')}$
\end{enumerate}
for all $u,u',v,v'\in C$.
Using that $dN_{L/F}(z)=g(c)$, (1) is the same as
$$g(c)(g(\overline{v'})g(v))=N(z^{-1})g(z)(( \overline{g(v')}\bar z)(zg(v))),$$
i.e. equivalent to
 \begin{enumerate}
\setcounter{enumi}{2}
\item
$g(\overline{v'})g(v)=N(z^{-1})( \overline{g(v')}\bar z)(zg(v)),$
for all $v,v'\in C.$
\end{enumerate}
Put $v=1$ then this yields $g(\overline{v'})=N(z^{-1}) (g(\overline{v'})\bar z),$ which is equivalent to
$g(\overline{v'})=(g(\overline{v'})\bar z) \bar z^{-1},$ i.e. to
$\bar zg(\overline{v'})=g(\overline{v'})\bar z$ for all $v'\in C$. Thus $z\in {\rm Comm}(C)=F$ and (2)
 is true for all $u,u',v,v'\in C$. Now let $m=z^{-1}$.
\\ In (ii), this is
 equivalent to
  \begin{enumerate}
\item $(g(\overline{v'})g(v))g(c)=(( \overline{g(v')}\bar z)(zg(v)))d,$
\end{enumerate}
and (2) for all $u,u',v,v'\in C$. The proof is now analogous to (i).
\end{proof}

\begin{corollary} Let $F$ have characteristic not 2.
Let $C$ be an octonion division algebra over $F$ and $A = \Cay (C,c)$ or $A = \Cay_r(C,c)$.
Suppose that $c$ lies in a quadratic field extension $L\subset C$ of $F$.
Let $\Psi: {\rm Aut}(A)\to {\rm Aut}(C),$ $ F\mapsto F|_C$.
\\ (i) The maps $G:A\to A$, $G((u, v)) = (g(u), \pm g(v) )$ for all $g\in {\rm Aut}(C)$ are the only isomorphisms
of $A$ unless $i\in F$ and $g(c)=\bar c=-c$. In that case also the maps
 $G((u, v)) = (g(u), \pm i g(v) )$  are isomorphisms of $A$ for all $g\in {\rm Aut}(C)$.
\\ (ii) ${\rm Aut}(A)/{\rm ker} (\Psi)$ is a
normal subgroup of ${\rm Aut}(D)$ and $ {\rm ker} (\Psi)=\{id,G_{-1}\}$ with $G_{-1}((u,v))=(u,-v)$, unless $i\in F$
and $\bar c=-c$.
If $i\in F$ and $\bar c=-c$, then $ {\rm ker} (\Psi)=\{id,G_{-1},G_i, G_{-i}\}$, with $G_{i}((u,v))=(u,iv)$
and $G_{-i}((u,v))=(u,-iv)$.
\\ (iii) The map
$\Psi: {\rm Aut}(A)\to {\rm Aut}(C),$ $ F\mapsto F|_C$ is surjective.
\end{corollary}

\begin{proof} (i) $G:A\to A$ is an isomorphism if and only if there are elements $m\in F^\times$,
 $g\in {\rm Aut}(C)$ such that $G((u, v)) = (g(u), g(v) m^{-1})$ and $c=m^2g(c)$.
 Either $c=m^2c$ and hence $m=\pm 1$, or $c=m^2\bar c$  by Theorem 32. In the latter case, it follows that
 $\bar c=-c$ and $m=\pm i$.
\\ (ii) Let $F\in {\rm Aut}(A)$, then $F$ stabilizes $C$, therefore $F|_C\in {\rm Aut}(C)$ and
 $\Psi: F\mapsto F|_C$ is a homomorphism with kernel ${\rm ker} (\Psi)$ which is a normal subgroup.
The rest is obvious.
 \\ (iii) is trivial.
 \end{proof}

\bigskip
{\it Acknowledgements:} The author would like to thank E Darp\"o for his comments and suggestions on an earlier version
of this paper.

\end{document}